\documentclass{amsart}

\usepackage[OT2,T1]{fontenc}
\usepackage{setspace}
\usepackage{amsfonts}
\usepackage{amsmath}
\usepackage{amsthm}
\usepackage{amssymb}
\usepackage{rotating, tabularx}
\usepackage{appendix}
\usepackage{tikz}
\usetikzlibrary{matrix,arrows}

\usepackage[colorlinks]{hyperref}

\theoremstyle{plain}
\newtheorem{Theorem}{Theorem}[section]
\newtheorem{Lemma}[Theorem]{Lemma}
\newtheorem*{Lemma*}{Lemma}
\newtheorem{Proposition}[Theorem]{Proposition}

\newtheorem{Corollary}[Theorem]{Corollary}

\theoremstyle{definition}
\newtheorem{Definition}[Theorem]{Definition}

\theoremstyle{remark}

\DeclareMathOperator{\ord}{ord}

\DeclareMathOperator{\Jac}{Jac}

\DeclareMathOperator{\loc}{loc}

\DeclareMathOperator{\im}{im}

\newcommand{\Q}{\mathbb{Q}}
\newcommand{\Z}{\mathbb{Z}}

\renewcommand{\div}{\textrm{div}}
\newcommand{\dr}{\textrm{dR}}
\DeclareMathOperator{\red}{red}
\newcommand{\F}{\mathbb{F}}

\theoremstyle{remark}
\DeclareMathOperator{\NS}{NS}

\DeclareMathOperator{\dR}{dR}
\DeclareMathOperator{\rk}{rk}
\DeclareMathOperator{\res}{Res}

\begin{document}
\title{An effective Chabauty--Kim theorem}
\author{Jennifer S. Balakrishnan}
\address{Jennifer S. Balakrishnan, Department of Mathematics and Statistics, Boston University, 111 Cummington Mall, Boston, MA 02215, USA}
\email{jbala@bu.edu}
\author{Netan Dogra}
\address{Netan Dogra, Department of Mathematics, Imperial College London, London SW7 2AZ, UK }
\email{n.dogra@imperial.ac.uk}

\begin{abstract} The Chabauty--Kim method is a method for finding rational points on curves under certain technical conditions, generalising Chabauty's proof of the Mordell conjecture for curves with Mordell--Weil rank less than their genus.
We show how the Chabauty--Kim method, when these technical conditions are satisfied in depth 2, may be applied to bound the number of rational points on a curve of higher rank. This provides a nonabelian generalisation of Coleman's effective Chabauty theorem.\end{abstract}

\date{\today}

\maketitle
\tableofcontents

\section{Introduction}
Chabauty's method \cite{chabauty} is one of the most powerful tools for studying the Diophantine geometry of curves of genus larger than 1. In its original form, it gives a proof of the Mordell conjecture for curves $X/\Q$ of genus $g$ whose Jacobians have Mordell-Weil rank less than $g$. The simple idea underlying the proof is to try to prove finiteness of the rational points of a curve $X$ with Jacobian $J$ by bounding the intersection of $X(\mathbb{Q}_p )$ and the $p$-adic closure of $J(\mathbb{Q})$ inside $J(\mathbb{Q}_p )$. 

This paper concerns two subsequent refinements of Chabauty's argument. The first, due to Coleman, is an effective version in the sense of giving a bound on the number of rational points. This amounts to replacing `soft analysis' (finding, on each residue disk of $X_{\mathbb{Q}_p }$, a nontrivial power series vanishing on $X(\mathbb{Q})$), with `hard analysis' (giving a bound on the number of zeroes of this power series). By bounding the number of zeroes of this power series, Coleman produces a bound on the size of $X(\Q)$.

The second, due to Kim \cite{kim:siegel, kim:chabauty}, gives a generalisation of Chabauty's method which replaces the Jacobian with a nonabelian cohomology variety with values in (finite-dimensional quotients of) a motivic fundamental group in the sense of Deligne \cite{deligne1989groupe}. As explained in the next section, Kim's method  produces a decreasing sequence of subsets $X(\Q _p )\supset X(\Q _p )_1 \supset X(\Q _p )_2 \supset \ldots \supset X(\Q )$. Conjecturally, $X(\Q _p )_n =X(\Q )$ for all $n\gg 0$. However in general it is not known that $X(\Q _p )_n $ is eventually finite. By work of Coates and Kim \cite{coates2010selmer}, we know unconditionally that $X(\Q _p )_n $ is finite for $n\gg 0 $ when $X$ is a curve whose Jacobian has complex multiplication. 
Recently, Ellenberg and Hast extended this result to give a new proof of Faltings' theorem for solvable covers of $\mathbb{P}^1 $ \cite{ellenberg2017rational}.

In this paper we only use the set $X(\Q _p )_2 $, which is much simpler to describe. In analogy with Coleman's original result, we bound the size of $X(\Q )$, under certain technical conditions, by bounding the size of $X(\Q _p )_2$.
Just as with the original effective Chabauty results, if one is careful, one can improve the bounds in various ways, but in the interest of simplicity, here we focus on the problem of finding an explicit bound on $X(\mathbb{Q}_p )_2$ which is polynomial in the genus. 

To explain our conditions more precisely, we introduce some notation. Let $X$ be a curve of genus $g>1$ over $\mathbb{Q}$, with $\rk \Jac (X)=r$. Define 
\[
\rho _f (J):=\dim \NS (\Jac (X_{\Q }))+\dim (\NS (\Jac (X_{\overline{\Q }} )^{c=-1}).
\]

Our finiteness results will be dependent on one of the following conditions being satisfied:
\begin{itemize}
\item Condition A :
$
r=g
$
and $\rho _f (J) >1$.
\item Condition B :
$r=g$ and 
\[
\dim H^1 _f (G_T ,H^2 _{\acute{e}t}(X\times X_{\overline{\Q }},\Q _p (1)))=0
\].
\end{itemize}
For a generic curve $X$, the rank of the N\'eron-Severi group of $J$ will be 1, and hence Condition A will not hold. However the condition that $\rho _f (J)>1$ still arises in many examples of interest. For example, if $X$ is a non-trivial cover of a curve of higher genus, or more generally if $J$ is isogenous to a product of two abelian varieties, then $\rho _f (J)\geq \rho (J) \geq 2$.
By contrast, it is very difficult to give examples when Condition $B$ is satisfied; however, as explained in \cite[Lemma 2.4]{qcpart2}, the latter part of Condition B is implied by a conjecture of Bloch and Kato \cite{bloch-kato}.

As in \cite[Proposition 1]{qcpart2}, one may prove the finiteness of $X(\mathbb{Q}_p )_2 $ if Condition A or Condition B holds. 
For $v$ a prime of bad reduction, we define $n_v \in \mathbb{Z}_{>0}$ to be the size of the image of $X(\Q _v )$ under $j_{2,v}$ (see the next section for a precise definition).

\begin{Theorem}\label{maintheorem}Let $X/\Q$ be a curve of genus $g > 1$ with good reduction at a prime $p \geq 3$, satisfying condition $A$ or condition $B$.  Let $\kappa _p =1+\frac{p-1}{p-2}\frac{1}{\log (p)}$. Then
\begin{enumerate}
\item 
$\# X(\Q ) <\kappa_p \left(\prod _{v\in T_0 }n_v \right)\# X(\mathbb{F}_p )(16g^3+15g^2-16g+10).$
\item
If $X$ is hyperelliptic and $p\neq 2g+1$, then
\[
\# X(\Q )< \kappa _p\left(\prod _{v\in T_0 }n_v \right)  ((2g+2)\# X(\mathbb{F}_p )+2g \# W(\mathbb{F}_p )+8g^3 +64g^2 +20g +16),
\]
where $W$ is the subscheme of Weierstrass points.
\end{enumerate}
\end{Theorem}

As will be explained in the next section, one may obtain bounds on the local constants $n_v$ in terms of the reduction data of the curve $X$ at $v$. It seems difficult to avoid the bounds obtained by the nonabelian Chabauty method depending on how bad the reduction of $X$ is at bad primes. For this reason, it is unclear the extent to which this theorem could directly be used to prove uniformity results in the manner of Stoll and Katz--Rabinoff--Zureick-Brown \cite{stoll2013uniform, katz2016uniform}.

However, in special cases, one can control the local factors to provide uniform bounds on the number of rational points of special families of curves. We illustrate this with the following corollary.

\begin{Corollary}\label{explicit_cor}
Let $X$ be a smooth projective hyperelliptic curve of genus $g$ with good reduction at 3 and potential good reduction at all primes. 
If the curve satisfies Condition A or Condition B, then
\[
\# X(\Q )<24g^3 + 228g^2 + 120g + 72.
\]
\end{Corollary}
An example of a hyperelliptic curve satisfying the hypotheses regarding the reduction type is given by 
\[
X:y^2 =x^n +k,
\]
where $n$ is a square-free positive integer prime to $6$ and $k$ is an integer prime to 3. If $n$ is composite, then $X$ also satisfies $\rho (J)>1$, and hence in this case the bound on the number of rational points will hold whenever $r=g$.

The method of proof of Theorem \ref{maintheorem} may also be used to bound the number of integral points on hyperelliptic curves, answering a question of \cite{balakrishnan2013p}.
\begin{Theorem}\label{integraltheorem}
Let $X$ be a smooth projective hyperelliptic curve of genus $g$ with good reduction at $p\geq 3$ and Mordell-Weil rank $g$. Suppose $X$ has a rational Weierstrass point $\infty $, let $Y:=X-\infty $, and let $Y(\Z )$ denote the set of integral points of $Y$ with respect to a minimal regular model. Then
\[
\# Y(\Z )< \kappa _p \left( \prod _{v\in T_0 }m_v \right)(8g^3+44g^2-34g+9+(2g+1)\# Y(\mathbb{F}_p )+(2g-1)\# W(\mathbb{F}_p ))
\]
if $g>1$ and 
\[
\# Y(\Z )<2\kappa _p \left( \prod _v m_v \right) \# Y(\F _p )
\]
if $g=1$, where the $m_v $ are local constants as in \cite{balakrishnan2013p}, and $W$ denotes the scheme of Weierstrass points not equal to $\infty $.
\end{Theorem}

To explain the method of proof, we briefly recall Coleman's proof of effective Chabauty \cite{coleman:chabauty}. There, he gave a bound for the number of zeroes of $G:=\int \omega $ in a residue disk $]\overline{z}[$ for $\omega $ a global differential. 
This bound is derived from understanding some piece of the Newton polygon of $G$: specifically, from bounding the length of the slope $-1$ segment of the Newton polygon. By \emph{length} of a segment, we take the usual convention: the length of the projection of the segment onto the $x$-axis. We recall the following classical result: 
\begin{Proposition}
Suppose the slope $\leq -1$ segment of the Newton polygon has endpoint $(n,m)$. Then $G$ has at most $n$ zeroes in $B(0,|p|)$.
\end{Proposition}
\begin{proof}
See e.g. \cite[IV.4]{koblitz2012p}.
\end{proof}
In particular, Coleman related the Newton polygon of $G$ to the zeroes of $\omega $ mod $p$, which can be bounded by elementary algebraic geometry.

The idea of the proof in the depth 2 case is similar. We want to bound the number of zeroes of a non-algebraic power series $G$ (from depth 2 Chabauty--Kim, see Proposition \ref{QC2input}) in a residue disk $]\overline{z}[$, or equivalently, understand the slopes of its Newton polygon. We would like to reduce this to a question about the slopes of something algebraic, but as $G$ involves double integrals, we have to replace simply taking the derivative by applying a more complicated differential operator $\mathcal{D}$. We show in Section \ref{sec:bounding} that for suitable `nice' differential operators, we can relate the Newton polygon of $G$ to the zeroes of $\mathcal{D}(G)$.
We then want to find a $\mathcal{D}$ that sends our power series $G$ to some algebraic function whose zeroes we can bound mod $p$. We give constructions of $\mathcal{D}$ in the general case, hyperelliptic case, and hyperelliptic and integral points case in the three subsequent sections.

\subsection*{Acknowledgements}
We thank Kevin Buzzard, Minhyong Kim and Jan Vonk for helpful suggestions.  Balakrishnan is supported in part by NSF grant DMS-1702196, the Clare Boothe Luce Professorship (Henry Luce Foundation), and Simons Foundation grant \#550023.

\section{Explicit Chabauty--Kim at depth 2}\label{sec:CKdepth2}

We begin with a brief review of a few essential results from the Chabauty--Kim method \cite{kim:siegel, kim:chabauty}. Associated to a pointed curve $(X,b)$ over $\Q$, with a good reduction outside a finite set $T_0$, and a prime $p$ of good reduction, we have a map
\[
j_n :X(\Q )\to H^1 (G_T ,U_n )
\] 
where $T:=T_0 \cup \{p \}$, $G_T$ is the Galois group of the maximal extension of $\Q$ unramified outside $T$, and $U_n $ is the maximal $n$-unipotent quotient of the $\Q_p $ pro-unipotent completion of $\pi _1 ^{\acute{e}t}(X_{\overline{\Q }},b)$. We also have local maps 
\[
j_{n,v}:X(\Q _v )\to H^1 (G_{\Q _v},U_n )
\]
for $v$ in $T_0 $  and 
\[
j_{n,p}:X(\Q _p )\to H^1 (G_{\Q _p },U_n )
\]
(for the definition of $H^1 _f (G_{\Q _p },U_n )$ see \cite{kim:siegel}). We define 
\[
X(\Q _p )_n := j_{n,p}^{-1}\left( \loc _p (\cap _{v\in T_0 }\loc _v ^{-1}(X(\Q _v ))) \right) \subset X(\Q _p ),
\]
where $\loc _v $ is the localisation map from $H^1 (G_T ,U_n )$ to $H^1 (G_v ,U_n )$. By construction, the set of rational points 
$X(\Q )$ is a subset of $ X(\Q _p )_n $ for all $n$.

The behaviour of the maps $j_{n,v}$ is fundamentally different depending on whether or not $v=p$. In the $v\neq p$ case, we have the following theorem, due to Kim and Tamagawa \cite{kim2008component}.
\begin{Theorem}[Kim-Tamagawa]
Let $v$ be a prime not equal to $p$. Then for all $n$, $\im (j_{v,n})$ is finite.
\end{Theorem}
In fact, one can bound the image of $j_{n,v}$ in terms of the reduction data of the curve as follows. Let $L$ be a finite extension of $K_v$ over which $X_{K_v}$ acquires stable reduction. Let $\mathcal{X}/\mathcal{O}_L$ be a regular semistable model, and let $V(\mathcal{X}_{k_L})$ denote the set of irreducible components of the special fibre. 
Since the model is regular, specialisation induces a well-defined map
\[
r_v :X(K_v )\to V(\mathcal{X}_{k_L}).
\]
For $v$ a prime of bad reduction, we define $n_v \in \mathbb{Z}_{>0}$ to be the size of the image of $X(\Q _v )$ under $j_{2,v}$ in $H^1 (G_v ,U_2 )$. 
\begin{Lemma} 
With notation as above,
\[
n_v \leq \im (r_v ).
\]
\end{Lemma}
A detailed proof of this lemma will appear in \cite{BD}. However, for the sake of completeness, we briefly indicate the method of proof. First, if $L|K $ is a finite extension, then it is easy to show that $H^1 (G_K, U_n )\to H^1 (G_L ,U_n )$ is injective, hence one reduces to the case where $X$ has stable reduction. In this case, one can use the description of the action of $G_L$ on $\pi _1 ^{\acute{e}t, (v')}(X_{\overline{K}},b)$ (the maximal prime-to-$v$ quotient of $\pi _1 ^{\acute{e}t}(X_{\overline{K}},b)$) in terms of the dual graph of a regular semistable model given in \cite{oda1995note} to deduce that if points $b_1 $ and $b_2 $ lie on a common irreducible component of $V(\mathcal{X}_{k_L })$, then the class of $[\pi _1 ^{\acute{e}t,(v')}(X_{\overline{K}};b_1 ,b_2 )] $ in $H^1 (G_L ,\pi _1 ^{\acute{e}t,(v')}(X_{\overline{K}},b_1 )$ is trivial. This straightforwardly implies the lemma.

The finiteness of the maps $j_{n,v}$ allows us to partition the set $X(\Q _p )_n $ as follows. We refer to a tuple $\alpha =(\alpha _v )_v $ in 
$\prod _{v\in T_0 }\im (j_{n,v})$ as a collection of \textit{local conditions}, and define 
\[
X(\Q _p )_{\alpha }:= j_{n,p}^{-1}\left( \loc _p (\cap _{v\in T_0 }\loc _v ^{-1}({\alpha _v })) \right) \subset X(\Q _p ).
\]
By construction, $X(\Q _p )_n$ is the disjoint union of the $X(\Q _p )_{\alpha }$ for $\alpha $ a collection of local conditions.
The bound in Theorem \ref{maintheorem} comes from a bound on $\# X(\Q _p )_{\alpha }$ in the case of $n = 2$, together with a bound on the number of local conditions, i.e. on the size of $\prod _{v\in T_0 }j_{2,v}X(\Q _v )$.

\subsection{Local structure at $p$}
The power series $G$ in the introduction is from the following result of \cite{qcpart2}:
\begin{Proposition}[{\cite[Prop. 5]{qcpart2}}]\label{QC2input}
Let $X/\Q$ be a curve of genus $g > 1$.  Suppose $X$ satisfies condition A or condition B.  Let $\omega _0 ,\ldots ,\omega _{2g-1}\in H^0 (X,\Omega (D))$ be differentials of the second kind forming a basis of $H^1 _{\dr}(X)$, where $D$ is an effective divisor.  Then, for all local conditions $\alpha $, there are constants $a_{ij}$ and $a_i$, a differential of the third kind $\eta $, and a function $h\in H^0 (X,\mathcal{O}(2D))$ such that
$$
X(\mathbb{Q}_p )_{\alpha }\subset \left\{ z\in X(\mathbb{Q}_p ):G(z)=0 \right\},$$ where $$G(z):=\sum _{0\leq i,j <2g} a_{ij}\int ^z _b \omega _i \omega _j +\sum _{0\leq i<2g}a_i \int ^z _b  \omega _i +\int ^z _b \eta +h(z).$$
\end{Proposition}

Hence to prove Theorem \ref{maintheorem}, it is enough to prove the following: for any curve $X$ of genus $g>1$, one can choose an effective divisor $D$ and a collection of differentials of the second kind $\omega _0 ,\ldots ,\omega _{2g-1}\in H^0 (X,\Omega (D))$ with the following properties:
\begin{itemize}
\item 
The $\omega _i $ form a basis of $H^1 _{\dr}(X)$.
\item For any differential of the third kind $\eta \in H^0 (X,\Omega(D))$ and $h\in H^0 (X,\mathcal{O}(2D))$, and any constants $a_i $ and $a_{ij}$, the number of zeroes of $G(z)$ is at most  
\[
\kappa _p \# X(\mathbb{F}_p )(16g^3+15g^2-16g+10).
\]

\end{itemize}

\subsection{Proof of Corollary \ref{explicit_cor}}
In this subsection we prove that Theorem \ref{maintheorem} implies Corollary \ref{explicit_cor}.
\begin{Lemma}
Let $v\neq p$ be a prime of potential good reduction. Then, for all $n$,the map
\[
j_{n,v}:X(\Q _v )\to H^1 (G_v ,U_n )
\]
is trivial.
\end{Lemma}
\begin{proof}
Let $K_w |\Q _v$ be an extension over which $X$ acquires good reduction.
Then the map
\[
j_{n,w}:X(K_w )\to H^1 (G_w ,U_n )
\]
has trivial image.
Recall from \cite[I.5.8]{serregc:1997} that, given a profinite group $G$, closed normal subgroup $H$, and $G$-group $A$, we get an exact sequence of pointed sets
\[
H^1 (G/H,A^H ) \to H^1 (G,A)\stackrel{\res }{\longrightarrow } H^1 (H,A).
\]
We apply this when $G=G_v $, $H=G_w$. We claim $U_n ^{G_w }=1$. To see this, note that it is enough to show that the graded pieces $U_n [i]$ of $U_n $ with respect to the central series filtration satisfy $U_n [i]^{G_w }=1$, which follows from the fact that $U_n [i]$ is an unramified representation of $G_w$ of weight $-i$.

Hence we deduce that the restriction map
\[
H^1 (G_v ,U_n )\stackrel{\res }{\longrightarrow } H^1 (G_w,U_n)
\]
is injective. The lemma thus follows from commutativity of the diagram
$$
\begin{tikzpicture}
\matrix (m) [matrix of math nodes, row sep=3em,
column sep=3em, text height=1.5ex, text depth=0.25ex]
{X(K_v) & H^1 (G_v ,U_n)  \\
 X(L_w) & H^1 (G_w,U_n ) \\ };
\path[->]
(m-1-1) edge[auto] node[auto] {$j_{n,v}$} (m-1-2)
edge[auto] node[auto] {} (m-2-1)
(m-1-2) edge[auto] node[auto] {$\res $} (m-2-2)
(m-2-1) edge[auto] node[auto] { $j_{n,w}$} (m-2-2);
\end{tikzpicture} $$
\end{proof}
Now let $X$ be as in Corollary \ref{explicit_cor}. Then all the $n_v$ are 1. The corollary follows from taking $p=3$ and using the Hasse--Weil estimate for $\# X(\F _3 )$.

\section{Bounding the number of zeroes via a differential operator}\label{sec:bounding}
In this section, we explain how to bound the zeroes of a power series $G$ by finding a bound on $\mathcal{D}(G)$ for a suitably `nice' (in a way we will make precise shortly) differential operator $\mathcal{D}$. The construction of a nice differential operator in the case when $G$ is the Coleman function from Proposition \ref{QC2input} will be given in the next section.

We begin by fixing notation. We denote by $v$ the $p$-adic valuation homomorphism $\mathbb{Q}_p ^\times \to \mathbb{Z}$.  We fix a point $b$ and a rational function $x$ which is a uniformizing parameter at $b$. We let $]b[$ denote the tube (or \emph{residue disk}) of $b$, i.e., the set of points reducing to $b$ mod $p$. Given an analytic function $F$ on $]b[$, we let $N_b (F)$ denote the number of $\mathbb{C}_p$-valued zeroes of $F$ in $]b[$.

Let $C_i $ denote the function 
$\mathbb{Q}_p [\! [x]\! ]\to \mathbb{Q}_p$
sending a power series to its $x^i $ coefficient.
By a differential operator, we will simply mean an element of the noncommutative ring $\mathbb{Q}_p [\! [x]\! ][\frac{d}{dx}]$. By an algebraic differential operator, we will mean a differential operator in the image of $K(X)[\frac{d}{dx}]$. The \textit{order} of a differential operator will refer to its  degree as a polynomial in $\frac{d}{dx}$, when given in the form $\sum_{i=0}^N a_i (\frac{d}{dx})^i $, for $a_i \in \mathbb{Q}_p [\! [x]\! ]$. 
\begin{Definition}
A differential operator $\mathcal{D}=\sum _{i=0}^N g_i \frac{d^i }{dx^i }\in \mathbb{Q}_p [\! [x]\! ][\frac{d}{dx}]$ is \textit{nice} if all the $g_i $ are in $\mathbb{Z}_p [\! [x]\! ]$, and $g_N$ is in $\mathbb{Z}_p [\! [x]\! ]^\times $. 
\end{Definition}
The main result of this section is the following proposition, which shows that one may use nice differential operators to bound the zeroes of power series, in analogy with Coleman's use of differentiation.
\begin{Proposition}\label{algebraic_implies_bound}
Let $G$ be a power series in $\mathbb{Q}_p [\! [x]\! ]$. Let $\mathcal{D}$ be a nice differential operator of order $N$. Suppose $\mathcal{D}(G)$ is an algebraic function with no poles on $]b[$. Then the number of zeroes of $G$ in $]b[$ is at most 
$\kappa _p (N_b (\mathcal{D}(G))+N)$. 
\end{Proposition}

The proof of this proposition will occupy the remainder of the section.

\begin{Lemma}\label{lemma:modp_bound}
Let $F\in \mathbb{Q}_p [\! [x]\! ]$ come from a nonzero element of $\mathbb{Q}_p (X)$ without poles in $]b[$. Then $\{ v (C_i (F)) :i\geq 0 \}$ is bounded below, and the least $i$ such that $v (C_i (F))$ attains this bound is less than or equal to $N_b (F)$.
\end{Lemma}
\begin{proof}
There is some $\lambda $ in $K$ such that $\lambda F$ reduces to a nonzero rational function on $X_{\mathbb{F}_p }$. Since $F$ has no poles in $]b[$, the reduction mod $p$ of $\lambda F$ is the $\red _p (x)$-adic expansion of $\lambda F$ thought of as a rational section of $X_{\mathbb{F}_p }$, hence the least $i$ such that the minimum of $v (C_i (F))$ is attained is just the order of $\red _p (\lambda F)$.
\end{proof}
Now let $G$ be a power series with $\mathcal{D}(G) \in H^0 (X,\mathcal{O}(D))$ as in  Proposition \ref{QC2input}. Let $M$ denote the length of the slope $\leq -1$ part of the Newton polygon. Write $\mathcal{D}$ as 
$\sum_{i=1}^N g_i (\frac{d}{dx})^i $, where $g_N \in \mathbb{Z}_p [\! [x]\! ]^\times $. Recall the following well-known lemma:
\begin{Lemma}\label{trivial}
For any $n_1 \leq n_2 $,
$$v \left(\frac{n_2 !}{n_1 !} \right)\leq \log _p (n_1 )+\frac{n_2 -n_1}{p-1}.$$
\end{Lemma}
\begin{proof}
Using Legendre's formula of $v(n!) = \frac{n - s(n)}{p-1},$ where $s(n)$ is the sum of digits in base $p$, it follows that
\begin{align*}v\left(\frac{n_2!}{n_1!}\right) = \frac{n_2-n_1}{p-1} + \frac{s(n_1)-s(n_2)}{p-1} &\leq \frac{n_2-n_1}{p-1} + \frac{s(n_1)}{p-1}\\
&\leq \frac{n_2-n_1}{p-1} + \frac{(p-1)\log_p(n_1)}{p-1}\\
&=\frac{n_2-n_1}{p-1} + \log_p(n_1).\end{align*}
\end{proof}

\begin{Lemma}\label{also_trivial}
Let $M$ be the length of the slope $\leq -1$ part of the Newton polygon of $G$. Suppose that $M>1$, and $i \leq M$ satisfies 
$$
v(C_i (G)) \leq v(C_M (G))+v (M!/i!).
$$
Then $i \geq \kappa _p ^{-1}M.$
\end{Lemma}
\begin{proof}
Since $i\leq M$ and $M$ is the length of the slope $\leq -1$ part of the Newton polygon, we have 
$$
M-i \leq v(C_i (G))-v(C_M (G)) \leq v (M!/i!) .$$
This implies 
$\log _p (i)+(M-i)/(p-1) \geq M-i$,
by the previous lemma. Using the inequality  $\log _p (i)\leq i/\log (p)$, we get 
$$
\kappa _p i \geq M.
$$
\end{proof}
Given a power series $F$, let $S(F)=\{ i \geq 0: v(C_i (F))=\min \{ v(C_j (F)):j \geq 0 \} \}$ if this minimum exists, and take $S$ to be empty otherwise.

We now prove a key lemma which gives a quantitative relation between the Newton polygon of $G$ and the Newton polygon of $\mathcal{D}(G)$, when $\mathcal{D}$ is a nice differential operator. The idea of the proof is as follows. 
Let $s$ denote the minimum of $v(C_i (\mathcal{D}(G)))$. We would like to say that if the valuation of $C_M (G)$ is smaller than the valuation of $C_i (G)$ for all $i<M$, then the valuation of $C_{M-N}(\mathcal{D}(G))$ is smaller than that of $C_{i}(\mathcal{D}(G))$ for all $i<M-N$ (and hence $s\geq M-N$). 

This is not quite true, because when we apply $(\frac{d}{dx}) ^N $ to $C_M (G)x^M$, we increase the valuation by $v(M!/(M-N)!)$, so it may happen that there is some cancellation. However, for such cancellation to occur, there must be an $M_1 <M$ for which $v(C_{M_1 } (G))$ is within $v(M!/(M-N)!)$ of $v(C_M (G))$. Similarly, if $v(C_{M_1 -N}(\mathcal{D}(G)))$ is not smaller than $v(C_i (\mathcal{D}(G)))$ for all $i<M_1 -N$, then there must be some $M_2 <M_1 $ such that $v(C_{M_2 }(G))$ is close to $v(C_{M_1 }(G))$, and so on giving a sequence $M,M_1 ,\ldots $ until we get to $M_n \leq s+N$. By construction, the $v(C_{M_i}(G))$ are `close together', but since $M$ is the endpoint of the slope $\leq -1$ part of the Newton polygon they are also `far apart', and comparing these two conditions gives the lemma. 

Note that, without any additional conditions on $\mathcal{D}$ or $G$, to prove a result of the form ``$\mathcal{D}(G)$ has small slopes implies $G$ has small slopes"  it is necessary to assume $p>2$ (consider, for example, the case $\mathcal{D}=(d/dx) -1$ and $G=\exp (x)+1$). Note that, by Lemma \ref{lemma:modp_bound}, if $F$ is algebraic without poles on $]b[$, then $\min S(F) \leq \ord _{\overline{b}}(\red _p (F))$. Hence the following lemma implies Proposition \ref{algebraic_implies_bound}.

\begin{Lemma}
Let $p>2$, and let $M$ be the length of the slope $\leq -1$ part of the Newton polygon of $G$.   Suppose $S(\mathcal{D}(G))$ is nonempty.
Then
$$
M < \kappa _p (N+\min S(\mathcal{D}(G))).
$$

\end{Lemma}
\begin{proof}
For integers $i\leq j$, let 
\[
q(i,j):=\left\{
\begin{array}{cc}
v ( i!/(j-N)!), &  \mathrm{if} \, i\geq j-N, \\
0, & \, \mathrm{otherwise}. \end{array} \right.
\]
 For $k\geq 0$, let 
$$
T(k)=\{0\leq i\leq k:v (C_i (G))+q(i,k) \leq v (C_k (G))+q(k,k) \}.
$$

Clearly $T(M)$ contains $M$. Note that the set of $i\in \mathbb{Z}_{\geq 0}$ for which $$v(C_{M-N}(\mathcal{D}(C_{i}(G)x^i )))\leq v(C_{M-N}(\mathcal{D}(C_{N}(G)x^N )))$$ is a subset of $T(M)$. Hence if $T(M)=\{M\}$, then $$v (C_{N-M}(\mathcal{D}(G)))=v (C_M (G))+q(M,M),$$ and either $M\leq N$ or 
$N-M=\min S(\mathcal{D}(G))$. In both cases the lemma follows.

Suppose now that $T(M)$ has cardinality larger than 1. We define a decreasing sequence $M_0 ,\ldots ,M_n $ of positive integers as follows. Let $M_0 :=M$, and define $M_1 =\min T(M_0 )$. If $T(M_1 )=\{ M_1 \}$, this is the end of the sequence, otherwise we define $M_2 $ as the minimum, and so on. Let $M_n $ be the last term in the sequence.

We claim $M_n \geq M-\frac{p-1}{p-2}\log _p (M_n)$.
To see this, note that for each $i$, we have 
$$
v(C_{M_{i+1}}(G))+q(M_{i+1},M_i) \leq v(C_{M_i }(G))+q(M_i ,M_i ).
$$
Hence 
$$
v(C_{M_n}(G))-v(C_{M_0}(G)) \leq \sum _i (q(M_i ,M_i )-q(M_i ,M_{i+1})) \leq v(M_0 !/M_n !).
$$
Since they lie in the slope $\leq -1$ part of the Newton polygon, this implies that $M_n$ and $M=M_0$ satisfy the inequality
$M-M_n \leq v(M! /M_n !)$, which by Lemma \ref{trivial} is less than or equal to  $\log _p (M_n)+(M-M_n )/(p-1)$. Thus we deduce 
$$
M_n \geq M-\frac{p-1}{p-2}\log _p (M_n).
$$

Hence, by Lemma \ref{also_trivial}, we have 
$$
M-\frac{p-1}{p-2}\log _p (M_n) \leq N+N_b (\mathcal{D}(G)).
$$
On the other hand, since $M\geq 2$ we have the elementary estimate 
$$\log _p (M_n) < M_n/\log (p).$$\end{proof}

\subsection{Example: integral points on elliptic curves}

Before describing a general method for constructing suitable nice differential operators,  we illustrate how Proposition \ref{algebraic_implies_bound} can be used to prove effective versions of known finiteness results in the quadratic Chabauty method by considering the simplest possible case: that of integral points on a rank 1 elliptic curve. By work of Kim \cite{Kim:rank1, BKK11}, we know that
integral points on rank 1 elliptic curves are contained in the zeroes of 
\[
G(z)=\int ^z _t \omega _0 \omega _1 +a \int ^z _t \omega _0 \omega _0 +b_i 
\]
for some constants $a,b_i \in \Q _p $, where the number of $b_i$ is determined by the Tamagawa numbers at bad primes. In this case, finding a differential operator is quite simple: if we take $\mathcal{D}= \left( \frac{d}{\omega _0 } \right) ^2 $, then
\[
\mathcal{D}(G)=x+a,
\]
and we deduce that $\# X(\Z _p )_2 < 2\kappa _p \left( \prod _v m_v \right) \# (E-O)(\F _p )$. 

\section{Differential operators for rational points: general case}\label{sec:diffops}
To use Proposition \ref{algebraic_implies_bound} to bound $X(\Q _p )_2 $, it remains to give a construction of a nice differential operator $\mathcal{D}$ such that $\mathcal{D}(G)$ is an algebraic function whose divisor can be controlled when $G$ is the iterated integral function from Proposition \ref{QC2input}.
The construction of the operator $\mathcal{D}$ is elementary.  First we make some preliminary notes about calculating the divisor of $\mathcal{D}(F)$ when $F$ and $\mathcal{D}$ are algebraic.
\begin{Lemma}\label{trivial_divisor}
Let $D=\sum n_i P_i $ be an effective divisor and let $F(x)$ be a function in $H^0 (X,\mathcal{O}(D))$. Suppose $dx$ is an algebraic differential with divisor $W-W'$, with $W,W'$ effective, and $W=\sum m_i Q_i $. Define $D_0 :=\sum P_i $ and $W_0 :=\sum Q_i $. 
Then, for all $j>0$,
\[
\frac{d^j F}{dx^j } \in H^0 (X,\mathcal{O}(jW+(j-1)W_0+D+jD_0 )).
\]
In particular $\frac{d^j F}{dx^j }\in H^0 (X,\mathcal{O}((2j-1)W+(j+1)D)$.
\end{Lemma}

\begin{proof}When $j=1$, this follows from the fact that
the differential $dF$ has poles only in the support of $]b[$ and has a pole of order $n_i+1$ at $P_i $.  The differential $dx$ only has zeroes at $W$, each of order 1. The general case follows by induction.
\end{proof}

We now restrict to our specific case of interest. Fix a point $\overline{z}$ in $X(\F _p )$. Let $D$ be an effective divisor on $X$ whose support is disjoint from $]\overline{z}[$. Let $\omega _0 ,\ldots ,\omega _{2g-1}\in H^0 (X,\Omega ^1  (D) )$ be a set of differentials of the second kind forming a basis of $H^1 _{\dR}(X)$. 
Let $x\in \mathbb{Q}_p [\! [t]\! ]$ be a formal parameter at some point $z_0 \in ]\overline{z}[$, such that $dx$ is algebraic with divisor $D_1 -D_0 $, (where $D_1 $ and $D_0 $ are effective). Let
$f_i := \omega _i /dx \in H^0 (X,\mathcal{O}(D+D_1 ) ).$ Finally, let $\eta $ be a differential in $H^0 (X,\Omega ^1 (D))$, and let
\[
G(z):=\sum a_{ij}\int ^z _b \omega _i \omega _j +\sum a_i \int ^z _b  \omega _i +\int ^z _b \eta +h(z)
\]
be the Coleman function from Proposition \ref{QC2input}.

The first step in constructing a differential operator satisfying the hypotheses of Proposition \ref{algebraic_implies_bound} is to reduce to constructing a nice differential operator which kills all the $f_i $.
\begin{Lemma}\label{lemma:annihilating}
Suppose $\mathcal{D}_1 =\sum _{i=0}^N g_i (\frac{d}{dx})^i $ is a nice differential operator of degree $N$, with coefficients in $H^0 (X,\mathcal{O}(E))$, for an effective divisor $E$ such that 
$$
\mathcal{D}_1 (f_i )=0
$$
for all i. Then $\mathcal{D} :=\mathcal{D}_1 \frac{d}{dx}$ is a nice differential operator with \[
\mathcal{D}(G)\in 
\begin{cases}& H^0 (X,\mathcal{O}(E+3(N-1)D_1 +(N+3)D)), N \geq 4 \\
& H^0(X, \mathcal{O}(E +(2N+1)D_1 + (N+3)D)), N = 2, 3.
\end{cases}
\]
\end{Lemma}

\begin{proof}
The operator $\mathcal{D}$ is nice because its leading coefficient is the same as that of $\mathcal{D}_1 $. We deal with the $\int \omega _i , \int \eta ,h$ and $\int \omega _i \omega _j $ terms of $G$ separately. First, we have 
\[
\mathcal{D}\left(\int \omega _i \right)=\mathcal{D}_1 (f_i )=0.
\]
For $\int \eta $, $\frac{d}{dx}(\int \eta )=\eta /dx \in H^0 (X,\mathcal{O}(D+D_1 ))$. Thus by Lemma \ref{trivial_divisor}, $\mathcal{D}(\int \eta )\in H^0 (X,\mathcal{O}(E+(2N+1)D_1 +(N+1)D))$.
For $h$, by Lemma \ref{trivial_divisor} we have, for all $k>0$,  
\[
\frac{d^k h}{dx^k }\in H^0 (X,\mathcal{O}((2k-1)D_1 +(k+2)D),
\]
hence  $\mathcal{D}(h)\in H^0 (X,\mathcal{O}((2N+1)D_1 +(N+3)D+E))$.  

Finally,
\begin{align*}
\mathcal{D}\left(\int \omega _i \omega _j \right) & =\mathcal{D}_1 \left(f_i \int \omega _j \right) \\
& = \sum _{k\leq N}g_k \left(\frac{d}{dx}\right)^k \left(f_i \int \omega _j \right) \\
& = \sum _{k\leq N}g_k \sum _{0\leq m\leq k}\binom{k}{m}\left(\frac{d}{dx}\right)^m \left(f_i \right)\left(\frac{d}{dx}\right)^{k-m}\left( \int \omega _j \right) \\
& = \mathcal{D}_1 (f_i )\int \omega _j +\sum _{k\leq N} g_k \sum _{0\leq m<k}\binom{k}{m}\left(\frac{d}{dx}\right)^m (f_i )\left(\frac{d}{dx}\right)^{k-m-1} (f_j ) \\
& = \sum _{k\leq N} g_k \sum _{0\leq m<k}\binom{k}{m}\left(\frac{d}{dx}\right)^m (f_i )\left(\frac{d}{dx}\right)^{k-m-1} (f_j ) \\
\end{align*}
since $\mathcal{D}(f_1 )=0$. By Lemma \ref{trivial_divisor}, for all $k\leq N$, and all $m<k$,
\[
\left(\frac{d}{dx}\right)^m (f_i )\left(\frac{d}{dx}\right)^{k-m-1} \left(f_j \right) \in H^0 (X,\mathcal{O}((2k-4)D_1 +(k+1)D_2 ) ),
\]
where $D_2 :=D_1 +D$. Hence $\mathcal{D}(\int \omega _i \omega _j )\in H^0 (X,\mathcal{O}(3(N-1)D_1 +(N+1)D+E)).$
\end{proof}

\subsection{Finding $\mathcal{D}_1 $: the general case}\label{subsec:construct_general}
By the previous lemma, to get a bound on the number of zeroes of $G$, we need to construct a nice differential operator (with algebraic coefficients we can control) which annihilates all the $f_i :=\omega _i /dx$. In general, given $m$ functions $F_1 ,\ldots ,F_m $, it is an elementary exercise to construct a nontrivial differential operator of order at most $m$ which annihilates all the $F_i $. Hence the nontrivial question is how to find a \textit{nice} differential operator. 

First we introduce some notation. Let $F_1 ,\ldots ,F_{2g}$ be elements of a formal power series algebra $K[\! [x]\! ]$. Let $S$ be  subset of $\mathbb{Z}_{\geq 0}$ of size $2g+1$.
Write $S=\{ n_1 ,\ldots n_{2g+1} \}$ with $n_i <n_{i+1}$. Let $A(S,F_1 ,\ldots ,F_{2g})$ denote the $2g \times (2g+1)$ matrix with entries in $K[\! [x]\! ]$ whose $(i,j)${th} entry is $\frac{1}{n_j !}\frac{d^{n_j} }{dx^{n_j} }(F_i )$. Let $A^{(j)}(S,F_1 ,\ldots ,F_{2g})$ denote the $2g \times 2g$ matrix obtained by deleting the $j$th column. Let $\mathcal{D}=\mathcal{D}_{S,F_1 ,\ldots ,F_{2g}} \in K[\! [x]\! ][\frac{d}{dx}]$ denote the differential operator
\[
\mathcal{D}_{S,F_1 ,\ldots ,F_{2g}}:=\sum _{i=1}^{2g+1}(-1)^{i+1} \frac{n_{2g+1}!}{n_i !} \det (A^{(i)})\frac{d^{n_i }}{dx^{n_i }}.
\]
We first note that $\mathcal{D}$ is always a differential operator which annihilates the $F_i$, then show that the set $S$ can be chosen so that $\mathcal{D}$ is nice.
\begin{Lemma}
For any choice of $S$, and all i,
\[
\mathcal{D}(F_i )=0.
\]
\end{Lemma}
\begin{proof}
For any power series $f$,
\[
\mathcal{D}_{S,F_1 ,\ldots ,F_{2g}}(f) = n_{2g+1}! \det \left(
\begin{array}{ccc}
\frac{1}{n_1 !}\frac{d^{n_1 }}{dx^{n_1 }}(F_1 ) & \ldots & \frac{1}{n_{2g+1}! }\frac{d^{n_{2g+1} }}{dx^{n_{2g+1} }}(F_1 ) \\
\vdots & \ddots & \vdots \\
\frac{1}{n_1 !}\frac{d^{n_1 }}{dx^{n_1 }}(F_{2g} ) & \ldots & \frac{1}{n_{2g+1}! }\frac{d^{n_{2g+1} }}{dx^{n_{2g+1} }}(F_{2g} ) \\
\frac{1}{n_1 !}\frac{d^{n_1 }}{dx^{n_1 }}(f ) & \ldots & \frac{1}{n_{2g+1}! }\frac{d^{n_{2g+1} }}{dx^{n_{2g+1} }}(f ) \\
\end{array}
 \right).
\]
When $f=F_i$, the matrix does not have full rank.
\end{proof}

We now apply this construction in our case of interest. Let $D,D_0 ,D_1 ,\omega _i ,f_i$ be as defined earlier in this section.
\begin{Lemma}
There exists an $S\in \mathbb{Z}_{\geq 0}^{2g+1}$ with $\max S \leq \deg (D)+2g-1$  such that $\mathcal{D}_{S,f_0 ,\ldots ,f_{2g-1}}$  is nice. 
\end{Lemma}

\begin{proof}
By construction, for any choice of $S$, the differential operator $\mathcal{D}_{S,f_0 ,\ldots ,f_{2g-1}}$ lies in $\mathbb{Z}_p [\! [x]\! ][\frac{d}{dx}]$, hence the only nontrivial condition is that the leading coefficient is in $\mathbb{Z}_p ^\times $. Note that 
\[
\frac{1}{n_i !}\frac{d^{n_i }}{dx^{n_i }}f|_{x=0}=C_{n_i }(f),
\]
hence requiring that the leading coefficient is in $\mathbb{Z}_p ^\times $ is equivalent to requiring that 
\[
\det (C_i (f_j ))_{1\leq i\leq 2g,0\leq j \leq 2g-1} \in \mathbb{Z}_p ^\times .
\]
Therefore, by definition, the least $N$ such that there exists a subset $S$ with $\max S\leq N+1$ for which $\mathcal{D}_{S,f_0 ,\ldots ,f_{2g-1}}$ is nice is exactly the least $N$ such that the $f_i$ remain $\mathbb{F}_p $-linearly independent after reduction mod $(p,x^N)$.  
Suppose that for all subsets $S$ of $\{ 0,\ldots ,N \}$ of size $2g+1$, $\mathcal{D}_{S,f_0,\ldots ,f_{2g-1}}$ is not nice. Then there is a nontrivial $\mathbb{F}_p$-linear combination of $\red _p \omega _0 ,\ldots ,\red _p \omega _{2g-1}$ which has a zero of order $N$. This gives an element of $H^0 (X_{\mathbb{F}_p },\Omega ^1 (D))$ with a zero of order $N$, which completes the proof.
\end{proof}
\begin{proof}[Proof of Theorem \ref{maintheorem} part (1)]
To complete the proof of Theorem \ref{maintheorem}, it remains to estimate the degree of the coefficients of a nice $\mathcal{D}:=\mathcal{D}_{S,f_0 ,\ldots ,f_{2g-1}}$. We follow the construction of basis differentials in \cite[\S 4.2]{darmon2015algorithms}. Let $P$ be a non-Weierstrass point of $X(\mathbb{Q}_p )$ whose mod $p$ reduction is different from $b$, and $h\in \mathbb{Q}_p (X)$ a non-constant function in $H^0 (X,\mathcal{O}((g+1)P))$. Let $\omega _0 ,\ldots ,\omega _{g-1}$ be a basis of $H^0 (X,\Omega ^1 )$. Define $\omega _{i+g}=h\omega _i$ for $0\leq i\leq g-1$. Then $(\omega _i )_{0\leq i \leq 2g-1}$ gives a basis of $H^1 _{\dr }(X)$. 

Let $\omega \in H^0 (X,\Omega^1 )$ be a differential which does not vanish mod $p$. Let $x$ denote the formal parameter on $]z[$ obtained by integrating $\omega $. Let $D_1 =(\omega )$. As above, define $f_i =\omega _i /dx \in H^0 (X,\mathcal{O}(D_1 +(g+1)P))$. 

	Since $D=(g+1)P$ above, we have an $S=\{ n_1 ,\ldots ,n_{2g+1} \}$ such that $\mathcal{D}_1 =\mathcal{D}_{S,f_0 ,\ldots ,f_{2g-1}}$ is nice and $\max S \leq 3g$. Hence the differential operator $\mathcal{D}:=\mathcal{D}_1 \frac{d}{\omega }$ has order at most $3g+1$. To apply Lemma \ref{lemma:annihilating}, it remains to estimate the degrees of the coefficients of $\mathcal{D}_{S,f_0 ,\ldots ,f_{2g-1}}$.  By Lemma \ref{trivial_divisor}, we have, for all $k\geq 0$, 
\[
\left( \frac{d^k }{dx^k } \right) f_i \in H^0 (X,\mathcal{O}(2kD_1 +(g+k+1)P).
\]
The $k$th coefficient of $\mathcal{D}$ is hence a sum of functions in 
$$
H^0 (X,\mathcal{O}(\sum _{1\leq i\leq 2g+1, i\neq k}(2n_i D_1 +(g+n_i +1)P))).
$$
Note that 
$$
\sum _{i \neq k}n_i \leq \sum _{i=g+1}^{3g}i = 4g^2 +g.$$
Hence the coefficients of $\mathcal{D}_1$ are in 
\[
H^0 (X,\mathcal{O}((8g^2 +2g)D_1 +3g(2g+1)P)).
\]

Applying Lemma \ref{lemma:annihilating} with $g \geq 2$, $N=3g$, 
$E = (8g^2 +2g)D_1 +3g(2g+1)P$,
and $D =(g+1)P$
we find that 
\begin{align*}
\mathcal{D}(G) \in  & H^0 (X,\mathcal{O}((8g^2 +2g)D_1 +3g(2g+1)(P)+3(3g-1)D_1 +3(g+1)^2 P)) \\
& = H^0 (X,\mathcal{O}((8g^2 +11g-3)D_1 +3(3g^2+3g+1)P)).
\end{align*}
Hence $\mathcal{D}(G)$ has degree at most 
\[(8g^2 +11g-3)(2g-2) +3(3g^2+3g+1) = 16g^3+15g^2-19g+9.\]
Applying Proposition \ref{algebraic_implies_bound}, we deduce that on each residue disk, $X(\Q _p )_{\alpha }$ has at most 
$$\kappa_p(16g^3+15g^2-16g+10)$$ points.
\end{proof}

\section{Differential operators for rational points: hyperelliptic case}
\subsection{The hyperelliptic case: non-Weierstrass disks}\label{subsec:nonW}
In this subsection, we prove the second part of Theorem \ref{maintheorem}. Let $X$ be a hyperelliptic curve of genus $g$ with good reduction at $p\neq 2g+1$ and Mordell-Weil rank $g$. The assumptions on $p$ imply that $X$ has a smooth model over $\mathbb{Z}_p$ of the form 
$$
y^2 =f(x)=x^{2g+2}+ a_{2g+1}x^{2g+1} + \cdots + a_0.
$$ 
Let $\omega _i $ be the differential $x^i dx/y$. We take as a basis of $H^1 _{\dr} (X)$ a subset of the $\Q $-span of the differentials $\omega _i :=x^i dx/y$, $0\leq i\leq 2g$. Hence, changing notation somewhat, we may write $G$ in the form
\[
G(z)=\sum _{0\leq i,j \leq 2g}a_{ij} \int ^z _b \omega _i \omega _j +\sum _{0\leq i\leq 2g}a_i \int ^z _b \omega _i +h(z).
\]
Let $\infty:=\infty^+ + \infty^-$ denote the degree two divisor of the two points $\infty^+, \infty^-$ above infinity.  Since all the $\omega _i $ are in $H^0 (X,\Omega ^1 ((g+1)\infty ))$, $h$ lies in $H^0 (X,\mathcal{O}(2(g+1)\infty ))$.

First let $\mathcal{D}_0 $ be the differential operator $\frac{d}{\omega _0 }$. Then define $G_1 := \mathcal{D}_0 G$. Hence 
$$
G_1 =\sum a_{ij}x^i \int \omega _j +\sum a_i x^i +h_1
$$
where $h_1 :=y\frac{d}{dx}h$.
Define $\mathcal{D}_1 =(\frac{d}{dx})^{2g+1}$, and $\mathcal{D}=\mathcal{D}_1 \mathcal{D}_0 $.
\begin{Lemma}
Let $W$ denote the degree $2g+2$ divisor of Weierstrass points.
Then the power series $\mathcal{D}(G)$ lies in $H^0 (X,\mathcal{O}((g+1)\infty +(4g+1)W))$.
\end{Lemma}
\begin{proof}
We have $\mathcal{D}(\int \omega _i )=0$ for all $i$, so it will be enough to prove this for $h$ and for $\int \omega _i \omega _j $.  Note that we have $\div(dx) = W - 2\infty, \div(y) = W - (g+1)\infty,$ and so $\div(\omega_0) = (g-1)\infty$. For $h$, we use Lemma \ref{trivial_divisor} to deduce $$\mathcal{D}_0 (h)\in H^0 (X,\mathcal{O}((3g+2)\infty)).$$
Since $dx$ has divisor $W - 2\infty$,
if $F \in H^0 (X,\mathcal{O}(n\infty +mW))$, a direct computation gives
\[
\frac{dF}{dx}\in \left\{ \begin{array}{cc}
H^0 (X,\mathcal{O}((n-1)\infty +(m+2)W)), & m>0,\\
H^0 (X,\mathcal{O}((n-1)\infty +W)), & m=0.\\ \end{array} \right.
\]
Hence $\mathcal{D}(h) \in H^0 (X,\mathcal{O}((g+1)\infty +(4g+1)W)).$ 

For $\int \omega _i \omega _j $, we have 
\begin{align}\label{eqij}
\mathcal{D}\left(\int \omega _i \omega _j\right) & = \mathcal{D}_1 \left(x^i \int \omega _j \right) \nonumber \\
& = \sum _{k=0}^{i}\binom{2g+1}{k}\frac{i!}{(i-k)!}x^{i-k}\left(\frac{d}{dx}\right)^{2g-k+1}\int (\omega _j ) \nonumber \\
& = \sum _{k=0}^{i}\binom{2g+1}{k}\frac{i!}{(i-k)!}x^{i-k}\left(\frac{d}{dx}\right)^{2g-k}\left(\frac{x^j}{y}\right). 
\end{align}

Now we have $$\frac{x^j}{y} \in H^0 (X,\mathcal{O}((j-g-1)\infty +W),$$ so 
$$\left( \frac{d}{dx} \right) ^{(2g-k)}\left(\frac{x^j}{y}\right) \in H^0 (X,\mathcal{O}((j-3g+k-1)\infty +(4g-2k+1)W)),$$ which gives that  $$x^{i-k}\left( \frac{d}{dx} \right) ^{(2g-k)}\left(\frac{x^j}{y}\right) \in H^0 (X,\mathcal{O}((i+j-1-3g)\infty +(4g-2k+1)W)),$$ and thus each summand of \eqref{eqij} is an element of $H^0 (X,\mathcal{O}((g-1)\infty +(4g+1)W))$.
\end{proof}

Since the degree of $((g+1)\infty +(4g+1)W)$ is $2g+2+(4g+1)(2g+2)=8g^2 +12g+4$, applying Proposition \ref{algebraic_implies_bound}, and summing over all non-Weierstrass residue disks, we find that the number of points of $X(\Q _p )_{\alpha }$ which reduce to non-Weierstrass points away from infinity is at most $$\kappa_p((8g^2 +12g+4)+(2g+2) \# (X-W-\infty)(\mathbb{F}_p )).$$
For the two points at infinity, we may apply the same analysis with the equation at infinity 
\[
y^2 =a_0 x^{2g+2}+a_1 x^{2g+1}+\ldots +1.
\]
We deduce that the number of points of $X(\Q _p )_{\alpha }$ which reduce to non-Weierstrass points is at most $$\kappa_p((16g^2 +24g+8)+(2g+2) \# (X-W)(\mathbb{F}_p )).$$

\subsection{The hyperelliptic case: Weierstrass points}\label{subsec:W}

The computation at Weierstrass disks is carried out in a manner similar to the method developed in Section \ref{sec:diffops}. The essential difference is that, instead of trying to find a new nice differential operator $\mathcal{D}_1 $ annihilating the $2g$ functions $\{ f_0 ,\ldots ,f_{2g-1} \}$ for each residue disk $]b[$, we find a differential operator $\mathcal{D}_1 $ which annihilates the $2g+1$ functions $\omega _i /\omega _0 ,$  ($0\leq i\leq 2g$) at all Weierstrass disks. The price paid for this is that the degree is slightly larger.
 Let $B\in M_{2g+1} (K(X))$ denote the matrix $$B=\left(\frac{1}{(2i)!
}\left( \frac{d}{\omega _0 } \right) ^{2i} x^j \right)_{0\leq i,j\leq 2g}.$$

\begin{Lemma}
For all Weierstrass points $z=(\alpha ,0) \in X(\mathbb{F}_p )$, $\det (B)\in \mathcal{O}_{\mathcal{X},z}^\times $.
\end{Lemma}
\begin{proof}
Clearly $\det (B)$ is defined at $z$, so it is sufficient to prove it is nonzero at $z$. First note that, since $x-\alpha $ has a zero of order $2$ at $z$, a linear combination of $1,x,\ldots ,x^{2g}$ can have a zero of order at most $4g$. Hence the $(2g+1) \times (4g+1)$ matrix $\left(\frac{d^i }{\omega _0 ^i }x^j |_{z} \right)_{0\leq i\leq 4g,0\leq j\leq 2g}$ has rank at least $2g+1$. On the other hand, for all odd $j$, $\frac{d^j }{\omega _0 ^j }x^i$ is an odd function with respect to the hyperelliptic involution, and hence vanishes at $z$. Hence $B|_{z}$ is invertible in $M_{2g+1}(\mathbb{F}_p )$.
\end{proof}

\begin{proof}[Proof of Theorem \ref{maintheorem} part (2)] 
We deduce that  we can apply the construction of Section \ref{subsec:construct_general} with $A$ taken to be the $(2g+1)\times (2g+2)$ matrix \[
\left( \left( \frac{d}{\omega _0  }\right) ^i x^j \right) _{0\leq j\leq 2g,i=0,2,4,\ldots ,4g,4g+1}.
\] The function $\left( \frac{d}{\omega _0 } \right)^i x^j $ lies in $H^0 (X,\mathcal{O}((gi+j)\infty ) )$, hence the differential operator $\mathcal{D}_1 =\mathcal{D}_{S, \omega_0/\omega_0, \omega_1/\omega_0, \ldots, \omega_{2g}/\omega_{0}}$
has coefficients in $$H^0 \left(X,\mathcal{O}\left(\left( g\sum _{i=1}^{2g}2i \right)+g(4g+1)+\sum _{j=0}^{2g}j\right)\right)=H^0 (X,\mathcal{O}(4g^3 +8g^2+2g)\infty )).$$
Define $\mathcal{D}:=\mathcal{D}_1 \mathcal{D}_0 $, where $\mathcal{D}_0 :=(d/\omega _0 )$.
Applying Lemma \ref{lemma:annihilating} with  
$E=(4g^3 +8g^2+2g)\infty$,
$N=4g+1$, $D_1 =(g-1)\infty $, $D=(g+1)\infty $, we deduce
\[
\mathcal{D}(G)\in  H^0 (X,\mathcal{O}((4g^3 +24g^2 -2g+4)\infty )).
\]
The number of points of $X(\Q _p )_{\alpha }$ on all Weierstrass disks is hence, by Proposition \ref{algebraic_implies_bound}, at most 
$$\kappa _p ((4g+2)\# W(\mathbb{F}_p )+2(4g^3+24g^2-2g+4)).$$ Combining with the bounds from the non-Weierstrass residue disks in the previous subsection, we find that 
\[
\# X(\Q _p )_{\alpha } \leq \kappa _p ((2g+2)\# X(\mathbb{F}_p )+2g \# W(\mathbb{F}_p )+8g^3 +64g^2 +20g +16).
\]
\end{proof}

\section{Integral points for hyperelliptic curves}
The proof of the $g>1$ case of Theorem \ref{integraltheorem} follows a similar strategy to the previous section.
Let $X$ be a hyperelliptic curve of genus $g>1$ with equation
\[
y^2 =f(x)=x^{2g+1}+a_{2g}x^{2g} + \cdots  + a_0
\]
and suppose the rank of the Jacobian of $X$ is equal to $g$. Then the set $X(\mathbb{Z}_p )_2 $ is partitioned into a disjoint union of sets $X(\mathbb{Z}_p )_{\alpha }$. 
Each set $X(\mathbb{Z}_p )_{\alpha } $ is contained in the set of zeroes of a Coleman function of the form
\[
G(z)=\sum _{0\leq i,j <2g} a_{ij}\int ^z _b  \omega _i \omega _j +\sum _{0\leq i<2g} a_i \int ^z _b \omega _i +h(z),
\]
where $h\in H^0 (X,\mathcal{O}(4g \infty ) )$, with $\infty $ now denoting the divisor of degree 1 consisting of the unique point at infinity. Let $W$ denote the degree $2g+1$ divisor of Weierstrass points away from infinity and define $\mathcal{D}_0 :=\frac{d}{\omega _0 }$. For non-Weierstrass points, we take the differential operator to be 
\[
\mathcal{D}=\left( \frac{d}{dx} \right) ^{2g}\mathcal{D}_0 .
\]
Since $\omega _0$ has a zero of order $(2g-2)$ at $\infty$, we have
\[
\frac{dh}{\omega _0 }\in H^0 (X,\mathcal{O}((6g-1)\infty ) ).
\]
Similar to the case of an even degree model, since $dx$ has divisor $W-3\infty $, if $F$ is in $H^0 (X,\mathcal{O}(nW+m\infty ) )$ then 
\[
\frac{dF}{dx} \in \left\{ \begin{array}{cc}
H^0 (X,\mathcal{O}((n+2)W+(m-2)\infty ) ), & n>0 \\
H^0 (X,\mathcal{O}(W+(m-2)\infty ) ), & n=0. \\ \end{array} \right.
\]

Hence $\mathcal{D}(h)$ is in $H^0 (X,\mathcal{O}((2g-1)\infty +(4g-1)W))$.

For the remaining term, note that 
\begin{equation}\label{eqn:another_divisor}
\left( \frac{d}{dx} \right) ^k \left(\frac{x^j}{y}\right) \in H^0 (X,\mathcal{O}((2k+1)W+(2j-2k-2g-1)\infty ) ). \\
\end{equation}
Since 
\begin{align*}
\mathcal{D}\left(\int \omega _i \omega _j \right) &= \left( \frac{d}{dx} \right) ^{2g}\left(x^i \int \omega _j \right) \\
& = \sum _{0\leq k<2g} \binom{2g}{k} \frac{i!}{(i-k)!}x^{i-k}\left( \frac{d}{dx} \right) ^{2g-k-1}\left(\frac{x^j}{y}\right),
\end{align*}
equation \eqref{eqn:another_divisor} implies
\[
x^{i-k}\left( \frac{d}{dx} \right) ^{2g-k-1}\left(\frac{x^j}{y}\right) \in H^0 (X,\mathcal{O}((4g-2k-1)W+(2j+2i-6g+1)\infty ) ).
\]
Hence $\mathcal{D}(\int \omega _i \omega _j )$ lies in $H^0 (X,\mathcal{O}((4g-1)W+(2g-3)\infty ))$.  Arguing as in Section \ref{subsec:nonW}, we deduce that the number of integral points of $X$ which reduce to non-Weierstrass points mod $p$ is bounded by
\[
\kappa _p \left( \prod _{v\in T_0 }m_v \right)(8g^2+4g-4+(2g+1)\# (Y-W)(\mathbb{F}_p )).
\]
\subsection{Differential operators at Weierstrass points}
Let $B\in M_{2g}(K(X))$ be the matrix $\left( \frac{1}{(2i)!}\left( \frac{d}{\omega _0 }\right) ^{2i} x^j \right) _{0\leq i,j<2g}$. 
As in Section \ref{subsec:W}, we may show $\det (B)$ is a unit at all points in $W$, and hence construct a differential operator $\mathcal{D}_1 $ from the matrix 
\[
A= \left( \frac{1}{(i)!}\left( \frac{d}{\omega _0 }\right) ^i x^j \right) _{0\leq j<2g,i=0,2,\ldots ,4g-2,4g-1}.
\]
Since $\omega _0$ has a zero of order $(2g-2)$ at $\infty$, we find that $$\left( \frac{d}{\omega _0 } \right) ^i x^j  \in H^0 (X,\mathcal{O}((i(2g-1)+2j)\infty )).$$
We deduce that the coefficients of $\mathcal{D}_1$ lie in
$H^0 (X,\mathcal{O}(8g^3+4g^2-6g+1)\infty ).$
Define $\mathcal{D}:=\mathcal{D}_1 \mathcal{D}_0 $, where $\mathcal{D}_0 :=(d/\omega _0 )$.
Applying Lemma \ref{lemma:annihilating} with $E =(8g^3+4g^2-6g+1)\infty$, $D_1 =2(g-1)\infty $, $D=2g \infty $ and $N=4g-1$, we deduce that $\mathcal{D}(G)$ lies in 
$H^0 (X,\mathcal{O}((8g^3+36g^2-38g+13)\infty ))$ 
Hence by Proposition \ref{algebraic_implies_bound}, we find that the number of integral points reducing to Weierstrass points is bounded by 
$$
\left( \prod _v m_v \right) \kappa _p (8g^3+36g^2-38g+13+4g \# W(\F _p )).
$$
We deduce a bound for the total number of integral points of 
\[
\kappa _p \left( \prod _{v\in T_0 }m_v \right)(8g^3+44g^2-34g+9+(2g+1)\# Y(\mathbb{F}_p )+(2g-1)\# W(\mathbb{F}_p )).
\]

\bibliography{bibECK}

\bibliographystyle{alpha} 

\end{document}